\documentclass[12pt]{amsart}
\usepackage{amsmath,amscd,amsthm,amsfonts,amssymb,amsxtra}
\usepackage[all]{xy}
\SelectTips{cm}{}
\newtheorem{thm}{Theorem}[section]  
\newtheorem*{un-no-thm}{Theorem}
\newtheorem{cor}[thm]{Corollary}     
\newtheorem{lem}[thm]{Lemma}         
\newtheorem{prop}[thm]{Proposition}  
\newtheorem*{cl}{Claim}

\newtheorem{bigthm}{Theorem}

\newtheorem{bigcor}[bigthm]{Corollary}

\theoremstyle{definition}
\newtheorem{defn}[thm]{Definition}   

\theoremstyle{definition}

\theoremstyle{definition}

\theoremstyle{remark}
\newtheorem{rem}[thm]{Remark}        

\newtheorem*{acks}{Acknowledgements}

\newtheorem{rems}[thm]{Remarks}
 
\newtheorem{ex}[thm]{Example}

\begin{document}
\title[The refined transfer]{The refined transfer, bundle structures and algebraic $K$-theory}
\date{\today}
\author{John R.\ Klein}
\address{Wayne State University, Detroit, MI 48202}
\email{klein@math.wayne.edu}
\author{Bruce Williams}
\address{University of Notre Dame,
Notre Dame, IN 46556 } 
\email{williams.4@nd.edu}
\begin{abstract}
We give new homotopy theoretic criteria
for deciding when a fibration with homotopy finite fibers admits
a reduction to a fiber bundle with compact topological
manifold fibers. The criteria lead to an
unexpected result about homeomorphism groups of manifolds.
A tool used in the proof   
is a surjective splitting of the assembly map 
for Waldhausen's functor $A(X)$.

We also give concrete examples of fibrations
having a reduction to 
a fiber bundle with compact topological
manifold fibers but  which
fail to admit a compact fiber smoothing. 
The examples are detected by
algebraic $K$-theory invariants.

We consider a refinement
of the Becker-Gottlieb transfer.
We show that a version of the axioms described by Becker and Schultz
uniquely determines the refined transfer for
the class of fibrations admitting a reduction to a fiber bundle
with compact topological manifold fibers.

In an appendix,
we sketch a theory of characteristic classes for
fibrations. The classes are primary obstructions to finding a
compact fiber smoothing.
\end{abstract}
\maketitle
\setlength{\parindent}{15pt}
\setlength{\parskip}{1pt plus 0pt minus 1pt}

\def\Top{\bold T\bold o \bold p}
\def\Sp{\text{\bf Sp}}
\def\vo{\varOmega}
\def\vs{\varSigma}
\def\smsh{\wedge}
\def\flush{\flushpar}
\def\id{\text{id}}
\def\dbslash{/\!\! /}
\def\codim{\text{\rm codim\,}}
\def\:{\colon}
\def\holim{\text{holim\,}}
\def\hocolim{\text{hocolim\,}}
\def\Bbb{\mathbb}
\def\bold{\mathbf}
\def\Aut{\text{\rm Aut}}
\def\cal{\mathcal}
\def\sec{\text{\rm sec}}
\def\gda{G\text{\rm -}\delta\text{\rm -}\alpha}

\setcounter{tocdepth}{1}
\tableofcontents
\addcontentsline{file}{sec_unit}{entry}

\section{Introduction}\label{intro}
Let 
$$
p\: E \to B
$$
be a fibration whose base space $B$ and
whose fibers have the homotopy type of a finite
complex.
The transfer construction of Becker and Gottlieb 
\cite{Becker-Gottlieb} associates
to $p$
a ``wrong way'' stable homotopy class 
$$
\chi(p)\:B_+ \to E_+
$$ 
such that the assignment $p \mapsto \chi(p)$ is
homotopy invariant and natural with respect to base change
(here $B_+$ denotes $B$ with the addition of a disjoint basepoint).
The transfer has shown itself to be an important tool in 
algebraic topology. For example, one of its early applications
was a simple proof of the Adams Conjecture \cite{Becker-Gottlieb_adams}.


A refinement of the transfer, also considered by Becker and Gottlieb
\cite[top of p.\ 115]{Becker-Gottlieb}, has recently surfaced
in the Dwyer, Weiss and Williams 
approach to fiberwise smoothing problems and the theory of higher 
Reidemeister torsion 
(see \cite{Dwyer-Weiss-Williams},\cite{Badzioch-Dora}, 
\cite{Igusa_book}).   

Let $E^+$ denote the disjoint union $E\amalg B$.
Then $E^+$ is a retractive space over $B$. The category 
of such spaces is the subject of fiberwise homotopy theory
(cf.\ \cite{Crabb-James}).
The associated stable homotopy category is thus
the study of  fiberwise stable phenomena (cf.\ \cite{May-Sigurdsson}).

The  {\it refined transfer} of $p$ is a certain fiberwise stable
homotopy class 
$$
t(p) \: B^+ \to E^+\, .
$$
The Becker-Gottlieb transfer $\chi(p)$ is obtained from 
$t(p)$ by collapsing the preferred sections $B\to E^+$ and $B\to B^+$
to a point.

Becker and Schultz \cite{Becker-Schultz} gave an axiomatic characterization
of  
the Becker-Gottlieb transfer
under the assumption that the fibration $p$ is fiber homotopy
equivalent to a fiber bundle with compact topological manifold fibers.
Their axioms involve naturality, normalization,
compatibility with products and additivity of the transfer.

\begin{defn}
If $p\: E \to B$
admits a fiber homotopy equivalence to a fiber bundle
with compact topological manifold fibers, 
then $p$ is said to have a {\it compact
\rm TOP \it reduction}.
 
If $p$ is fiber homotopy equivalent
to a fiber bundle with compact  smooth manifold fibers,
then $p$ is said to have 
a {\it compact \rm DIFF \it reduction} or a {\it compact fiber smoothing}.
\end{defn}

\begin{rem}
The fibers
of these bundles are permitted to have non-empty boundary.
Our terminology in the smooth case differs from
that of Casson and Gottlieb \cite{Casson-Gottlieb}, who
instead use the term {\it closed fiber smoothing.} 
Our preference is to use `compact' instead of `closed' 
so as to avoid potential confusion.
\end{rem}

The following, communicated to
us by Goodwillie, gives fibrations with
homotopy finite fibers which fail to admit a compact
TOP reduction.

\begin{ex} \label{good_ex} Let $F$ be a connected based finite complex 
equipped with a based 
self homotopy equivalence $\theta\: F\to F$. Assume $\theta$
induces the identity map on fundamental groups and has
non-trivial Whitehead torsion. Then the mapping torus
$F\times_{\theta} S^1 \to S^1$, converted into a fibration,
does not admit a compact TOP reduction.
\end{ex}

The example is verified by contradiction.
A compact TOP reduction would  
yield a homotopy equivalence $k\: F\to M$, with $M$ 
a compact topological manifold,
and a homotopy inverse $k^{-1}\: M \to F$ such that the
composite $k\circ \theta\circ k^{-1} \: M \to M$ 
is homotopic to a homeomorphism.
But this would show that the torsion of 
$k\circ \theta \circ k^{-1}$ is trivial \cite[th.\ 1]{Chapman_invariance}. 
Since $\theta$ induces the identity on $\pi_1$,
the composition 
formula \cite[22.4]{Cohen_simple}, 
shows that the torsion of $\theta$ is also trivial. This
gives the contradiction. For specific  homotopy 
equivalences $\theta$
satisfying  \ref{good_ex}, see \cite[24.4]{Cohen_simple}.
\medskip

The main result of this paper is to give explicit homotopy theoretic criteria for
deciding when a fibration admits
a compact TOP reduction. Our approach is to use the recent work of Dwyer, Weiss and Williams,
specifically, the ``Converse Riemann-Roch Theorem'' which gives an abstract characterization
when such a reduction exists \cite{Dwyer-Weiss-Williams}, and
entails an understanding of how the refined transfer relates to Waldhausen's algebraic 
$K$-theory of spaces. Along the way, we will extend the Becker-Schulz axioms to the fiberwise
setting and show how the axioms characterize the refined transfer for those fibrations
admitting a compact TOP reduction. 
The proof of this characterization follows along the lines of Becker-Schulz, 
and we do not claim any originality in this direction.
As to whether the axioms characterize the refined transfer
for all fibrations with homotopy finite fibers is an interesting open question.

The axiomatic characterization of the refined transfer is independent of the rest of the paper
and is included because of Igusa's recent progress on axiomatizing higher Reidemeister
torsion invariants \cite{Igusa_axioms}. 
There is a close relationship between higher torsion and the refined
transfer: when the fibration is fiber homotopy equivalent to a smooth
fiber bundle with compact fibers, then the refined transfer
admits a lift into a group closely associated
with algebraic $K$-theory, and  this lift coincides with 
the higher torsion invariant of Dwyer, Weiss and Williams.
A currently unsolved problem is to determine whether the
Dwyer-Weiss-Williams torsion
coincides with Igusa's torsion. The problem would be
solved if one could verify that Igusa's axioms hold for
the Dwyer-Weiss-Williams torsion.
Some evidence in favor of this is that the axioms
we will shortly give for the refined transfer
seem to be close in spirit to Igusa's axioms, although further effort will be needed to 
pin down the exact relationship.

Here are the axioms.

\begin{defn} 
A {\it refined transfer}
is a rule, which assigns to fibrations
$p\: E \to B$ with homotopy finite base and fibers,
a fiberwise stable
homotopy class 
$$
t(p)\: B^+ \to E^+
$$
that is
subject to the following axioms:
\begin{itemize}
\item {\bf A1} (Naturality).  For commutative homotopy pullback diagrams
$$
\xymatrix{
E' \ar[r]^{\tilde f}\ar[d]_{p'} & E\ar[d]^{p}\\
B' \ar[r]_{f} & B
}
$$
in which $p'$ and $p$ are fibrations, we have
$$\tilde f^+\circ f_\ast t(p') \,\, = \,\, t(p)\circ f^+\, ,$$
where 
$f^+$
denotes  $f\amalg \id_B$, $\tilde f^+$ denotes $\tilde f \amalg \id_B $ and
$f_\ast t(p')$ is the effect of making $t(p')$ into a  fiberwise stable homotopy class
 $B$ by taking cobase change along $f$.

\item {\bf A2} (Normalization). Let $1\: B\to B$ be the identity. Then 
$$
t(1)\: B^+ \to B^+
$$
is the identity.

\item {\bf A3} (Products). For a product fibration $p\times p'\: E\times E' \to B \times B'$, we have
$$
t(p\times p')\,\,  =\,\,  t(p)\, \smsh \, t(p')\, . 
$$
where $\smsh$ means external fiberwise smash product.

\item {\bf A4} (Additivity). If  
$$
\xymatrix{ 
E_{\emptyset} \ar[r]^{j_2} \ar[d]_{i_1} & E_2 \ar[d]^{i_2} \\
E_1 \ar[r]_{j_1} & E
}
$$
is a commutative homotopy pushout diagram 
of fibrations over $B$, then 
$$
t(p) = (j_1)_* t(p_1) + (j_2)_* t(p_2) - (j_\emptyset)_*t(p_{\emptyset})\, ,
$$
where for  $S\subsetneq \{1,2\}$, 
$p_S \: E_S \to B$ denotes the projection and
$(j_S)_*\: E^+_S \to  E^+$ is the evident map.
\end{itemize}
\end{defn}

In Section \ref{construct}, we explain Becker and Gottlieb's construction of
a refined transfer. Their version will be called {\it the} refined transfer, employing
the definite article to distinguish it from other constructions satisfying the axioms.

\begin{bigthm} \label{first} Let $t$ and $t'$ be 
refined transfers defined on the class of fibrations
having homotopy finite fibers.
Then $t = t'$ for those fibrations which admit a compact {\rm TOP} reduction.
\end{bigthm}

We now give  homotopy
theoretic criteria for deciding when 
a fibration admits a compact TOP reduction. One should regard this
as the main result of the paper.

\begin{bigthm} \label{second} Let $p \: E \to B$ be a fibration
with homotopy finite base and fibers. Assume 
\begin{itemize}
\item $p$ comes equipped with a section,
\item $p$ is  $(r+1)$-connected and
\item $B$ has the homotopy
type of a cell complex of dimension $\le 2r$.
\end{itemize} 
Then $p$  admits a preferred compact {\rm TOP} reduction.
\end{bigthm}

\subsection*{Consequences}
Combining Theorems \ref{first} and \ref{second}, we
immediately obtain

\begin{bigcor} Let $t$ and $t'$ be refined transfers. The $t = t'$
for the fibrations  
appearing in Theorem \ref{second}.
\end{bigcor} 

Here is a way to construct examples 
satisfying Theorem \ref{second}. Start with any
Hurewicz fibration $p\: E \to B$ with 
homotopy finite base and fibers. 
The {\it (unreduced) fiberwise suspension} of $p$ is the
fibration $S_B p\:S_B E\to B$ whose total space is the double mapping cylinder
of the map $p$:
$$
S_B E = (B \times 0) \cup (E\times [0,1])\cup (B\times 1)
$$
(cf.\ \cite{Strom}).
The fiber of $S_B p$ at $b\in B$ is given by the unreduced suspension 
of the fiber of $p$ at $b$. Consequently, the connectivity of the map
$S_B p$ is one more than that of $p$, so iteration of the fiberwise suspension
construction eventually yields a fibration which satisfies the conditions of Theorem \ref{second}.

\begin{bigcor} Stably, any fibration $p\:E\to B$ with 
homotopy finite base and fibers  admits a compact 
{\rm TOP} reduction. That is, there is an iterated fiberwise suspension
$S_B^j p\: S_B^j E \to B$ which admits a compact {\rm TOP} reduction.
\end{bigcor}

The method of proof of Theorem \ref{second} 
yields a new and unexpected result about automorphism
groups of manifolds. For a compact connected manifold $M$ with
basepoint $*$ in its interior,
let $\text{TOP}(M,*)$ be the simplicial group whose 
$k$-simplices are the homeomorphisms of $\Delta^k \times M$
which commute with projection to $\Delta^k$ and 
which are the identity when restricted to $\Delta^k\times *$. 
Let $G(M,*)$ be defined similarly, using homotopy equivalences
in place of homeomorphisms. The forgetful homomorphism induces
a map of classifying spaces
$$
B\text{TOP}(M,*) \to BG(M,*)
$$ 
The surprise will be that
this map has a section up to homotopy along the $2r$-skeleton
of $BG(M,*)$ when  $M \subset {\Bbb R}^m$ 
is an $r$-connected compact codimension zero manifold
with a sufficiently low dimensional spine (the exact dimensions will
be spelled out in Section \ref{TtoE}).

More precisely, define the {\it stable homeomorphism group} 
$$
\text{TOP}^{\text{st}}(M,*)\, ,
$$
to be
colimit of $\text{TOP}(M\times I^k,*)$
via stabilization given by taking cartesian product with 
the unit interval. 

Similarly, one can define $G^{\text{st}}(M,*)$, but in
this case the associated inclusion 
$G(M,*) \to G^{\text{st}}(M,*)$ is a homotopy equivalence.
It follows that one has
a map of classifying spaces
$$
B\text{TOP}^{\text st}(M,*) \to BG(M,*) \, .
$$

\begin{bigthm}\label{decomp} 
Let $M\subset {\Bbb R}^m$ be a compact codimension zero
smooth submanifold. Assume $M$ is $r$-connected.

Then there is 
a space $X_M$ and a map $X_M \to B\text{\rm TOP}^\text{{\rm st}}(M,*)$
such that the composite
$$
X_M \to B\text{\rm TOP}^\text{{\rm st}}(M,*) \to BG(M,*)
$$ 
is $2r$-connected.

Furthermore, there is a preferred decomposition of homotopy groups
$$
\pi_*(\text{\rm TOP}^\text{{\rm st}}(M,*)) \cong \pi_*(G(M,*)) \oplus 
\pi_*(\text{\rm map}(M,\text{\rm TOP}))
\oplus \pi_{*+1}(\text{\it Wh}^{\text{\rm top}}(M))
$$
which is valid in degrees
$
* \le 2r-2
$. 
\end{bigthm}

In the above,
$\text{\it Wh}^{\text{\rm top}}(M)$
is the topological Whitehead space (\cite[\S3]{Wald_LNM},
\cite{Hatcher}), ${\rm TOP}$  is the group of homeomorphisms of 
euclidean space stabilized with respect to dimension, and 
$\text{map}(M,\text{\rm TOP})$ is the function space of maps
$M \to \text{\rm TOP}$.

\subsection*{Examples}
We now give
examples of fibrations which fail
to have a compact fiber smoothing but which do
admit a compact TOP reduction.

\begin{bigthm} \label{third} There exist  fibrations $p\:E \to B$ 
which admit a compact {\rm TOP} reduction
but which do not have a compact fiber smoothing.

The fibers of these fibrations have the homotopy type of 
a finite wedge of spheres $\vee_k S^n$, for suitable choice of $k$ and $n$.

Furthermore, these examples are detected in 
the rationalized algebraic $K$-theory of the integers.
\end{bigthm}

\begin{bigthm} \label{third_add}
Let 
$$
S^3 \to E \overset p\to S^3
$$
be the  spherical fibration corresponding to
the generator  of 
$$
\pi_3(BF_3) \cong \pi_5(S^3) =  {\Bbb Z}_2\, ,
$$ 
where $F_3$ is the topological monoid of based self homotopy 
equivalences of $S^3$.

Then $p$ admits a compact {\rm TOP} reduction but does 
 not admit a compact fiber smoothing.\end{bigthm}

The following result shows that the obstructions to 
compact fiber smoothing are killed when taking the
cartesian product with finite complex having trivial
Euler characteristic.

\begin{bigthm} \label{kill}
Let $p: E\to B$ be a fibration with homotopy finite fibers.
Let $X$ be a finite complex with zero Euler characteristic. 
Then the fibration
$$
q\: E \times X \to B
$$
given by $q(e,x) = p(e)$ admits a compact fiber smoothing.
\end{bigthm}

\begin{rem} At the time the first draft of this paper was written,
it was forgotten by both authors that this result was already stated in 
\cite[Cor.\ 5.2.5]{WW-Auto} with a sketched proof. 
This paper contains a different proof.
\end{rem}

\subsection*{The trace map}
Given a fibration $p\: E\to B$,
let $p^+: E^+ \to B^+$ be the associated map of retractive
spaces over $B$.
 
Given a refined transfer $t(p)\: B^+ \to E^+$,
we take its composition with $p^+$ to obtain a fiberwise stable
homotopy class 
$$
p^+\circ t(p)\: B^+ \to B^+\, .
$$ 
A straightforward unraveling of definitions
shows that $p^+\circ t(p)$ is equivalent to  
specifying an ordinary stable cohomotopy class
$$
\text{tr}_t(p) \: B_+ \to S^0
$$
(because $B^+$ coincides with $B\times S^0$).
The latter is  called the {\it trace} of the fibration $p$. 
(compare Brumfiel and Madsen 
\cite[p.\ 137]{Brumfiel-Madsen}).

The following is a uniqueness result about the trace.

\begin{bigthm}\label{trace} Let
$t$ and $ t'$ be refined transfers. 
Then $\text{\rm tr}_t = \text{\rm tr}_{t'}$ on the class
of fibrations whose base and fiber have the homotopy type of a finite complex.
\end{bigthm}

\begin{rem} For further results, see
Douglas \cite{Douglas} and Dorabia\l a and Johnson.
\cite{Dora-Johnson} 
\end{rem} 

\subsection*{Assembly}
The proof of Theorem \ref{second} uses the assembly map
for Waldhausen's algebraic of $K$-theory of spaces functor $A(X)$.
If $f$ is a homotopy functor from spaces to spectra, the assembly map
is a natural transformation 
$$
f^\%(X) \to f(X)
$$
which best approximates   
$f$ by an excisive functor $f^\%$ in the homotopy category of functors 
(recall that a functor is excisive if it preserves homotopy pushouts).

The crucial result used in the proof
of Theorem \ref{second} is a functorial stable range
splitting for the assembly map for $A(X)$ considered as
a functor on the category of based spaces.

\begin{bigthm} \label{split} For based spaces $X$, 
the assembly map
$$
A^\%(X) \to A(X)
$$
is stably split. 

More precisely, there is a homotopy
functor $X \mapsto B(X)$ from based spaces to spectra,
and a natural transformation
$B(X) \to A^\%(X)$ such that the composite map
$$
B(X) \to  A^\%(X) \to A(X)
$$
is $2r$-connected whenever $X$ is $r$-connected.
\end{bigthm}

Given what is already known about $A(X)$, 
this result is not hard to prove. However, 
it is worth stating here
since it is one of our main tools. 
The role of the basepoint here is crucial; the result is false
on the category of unbased spaces.

\begin{acks} Both authors wish to thank Bill Dwyer for
the method which produces the examples 
appearing in Theorem \ref{third}. We also wish to thank 
Ralph Cohen for miscellaneous discussions.

An earlier draft of this paper asserted that the splitting
in Theorem \ref{split} was valid on the category of unbased spaces. 
We are indebted to a referee for explaining to us why this is
not true.

The first author was partially supported by
NSF grants DMS0503658 and DMS0803363.
\end{acks}

\section{Conventions}\label{prelim}

\subsection*{Spaces}
We work in the category of compactly generated spaces.
A map of spaces $f\: X \to Y$ is a {\it weak homotopy equivalence}  
when it induces an isomorphism on homotopy in each degree.
A  space $X$ is {\it $r$-connected}, if
for every integer $k$ such that $-1 \le k \le r$,
every map $S^k \to X$ extends to map $D^{k+1}\to X$. 
In particular, every non-empty
space is $(-1)$-connected. The empty space is considered to be
$(-2)$-connected. A map of spaces is $r$-connected if
its homotopy fibers (with respect to all basepoints) are $(r-1)$-connected.
A space is {\it homotopy finite} if it has the homotopy type
of a finite cell complex.

Although Quillen model categories are barely mentioned in this paper,
we will be implicitly working in the model structure for spaces
whose weak equivalences are the weak homotopy equivalences, 
whose fibrations are Serre fibrations, and whose cofibrations
are the Serre cofibrations.

There is one notable exception to this policy: it is 
not known whether the fiberwise suspension
of a Serre fibration is again a fibration, but the analogous statement
is true in the Hurewicz fibration case.
So unless otherwise mentioned, we usually work with Hurewicz
fibrations.

A commutative square of spaces (or spectra)
$$
\xymatrix{
A \ar[r] \ar[d] & C\ar[d]\\
B \ar[r] & D
}
$$
is said to be {\it $r$-cartesian} if the map  
from $A$ to the homotopy pullback $P := \text{ holim } (B\to D\leftarrow C)$
is $r$-connected. More generally, suppose the square only commutes
up to homotopy. Given a choice of homotopy  $A\times [0,1] \to D$,
one gets a preferred map  $A \to P$.
In this instance, we say that the square together with its commuting homotopy
is $r$-cartesian provided that  $A\to P$ is $r$-connected.

\subsection*{Fiberwise spaces}
We will be assuming familiarity with fiberwise homotopy theory in
its unstable and stable contexts. 
The book of Crabb and James \cite{Crabb-James} gives the foundational material
on this subject.

For a cofibrant space $B$, let $T(B)$ denote the
category of {\it spaces over $B$}. An object of $T(B)$ consists of a space
$X$ and a reference map $X\to B$ (where 
the latter is typically suppressed from the notation).
A morphism $X\to Y$ is a map of spaces which is compatible with reference
maps. A morphism is a fibration or weak equivalence if and only if
it is one when considered as a map of topological spaces. This comes from
a model structure on $T(B)$, where the cofibrations are defined using the 
lifting property \cite{Quillen}.

The `pointed' version of $T(B)$ is the category $R(B)$ of 
{\it retractive spaces over $B$}. This category has objects consisting
of a space $Y$ together with maps $B\to Y$ and $Y \to B$ such that the 
composite $B\to Y \to B$ is the identity map. A morphism is a map
of underlying spaces which is compatible with the structure maps.
Again $R(B)$ is a model category by 
appealing to the forgetful functor
to spaces. An object of $R(B)$ is said to be {\it finite} if it
is obtained from the zero object by attaching a finite number of cells.

The {\it (reduced) fiberwise suspension} functor $\Sigma_B\: R(B) \to R(B)$
is given by mapping an object $Y$ to the object 
$$
\Sigma_B Y = (Y \times [0,1]) \cup_{X\times [0,1]} X \, ,
$$
where $X\times [0,1] \to Y \times [0,1]$ arises from the structure map
$X\to Y$ by taking cartesian product with the identity map and
$X\times [0,1]\to X$ is first factor projection.

$R(B)$ also has smash products. If $Y$ and $Z$ are objects, then
their fiberwise smash product is given by
$$
Y\smsh_B Z \,\, := \,\, (Y \times_B Z) \cup_{(Y\cup_X Z)} X
$$
where $Y \times_B Z$ is the fiber product
and $Y \cup_B Z \to Y \times_B Z$ is the evident map.
Note that the special case of $Z = S^1 \times B$ gives
$\Sigma_B Y$.

S.\ Schwede 
\cite{Schwede} has shown that the category of 
fibered spectra over $B$, i.e., 
spectra formed using objects of $R(B)$,
again forms a model category, where the weak
equivalences in this case are  the `stable weak homotopy equivalences.'

The recent book of  May and  Sigurdsson
equips the category of
fibered spectra over $B$ with a well-behaved internal smash product
\cite[\S11.2]{May-Sigurdsson}.

\section{Construction of a refined transfer}
\label{construct}

Becker and Gottlieb define a
refined transfer in their paper \cite[\S5]{Becker-Gottlieb}. 
The purpose of this section is to sketch the idea behind
their construction.

First consider the case
when $B$ is a point. Let $F$ be a homotopy finite space,
which for convenience we take to be  cofibrant.
We let $F_+$ denote the effect of adding a disjoint basepoint to $F$.
The {\it $S$-dual} of $F_+$ is the spectrum $D(F_+)$ which 
is the mapping spectrum $\text{map}(F_+,S^0)$, where $S^0$ is
the sphere spectrum.
Explicitly, it is the spectrum whose $k$-th space is the space of
maps $F_+ \to QS^k$, where
$Q = \Omega^\infty\Sigma^\infty$ is the stable homotopy functor. 
More generally, we use the notation
$D(E)$ for the function spectrum of maps $E\to S^0$ whenever
$E$ is a homotopy finite spectrum.

There is a map of spectra
$$
d\: F_+ \smsh D(F_+) \to S^0
$$
which defined as the adjoint to the identity map of $D(F_+)$.
The canonical stable map 
$$
F_+ \to D(D(F_+))
$$
is a weak equivalence. Furthermore, we have a
preferred weak equivalence
$$
F_+ \smsh D(F_+) \,\, \simeq \,\, D(F_+ \smsh D(F_+)) 
$$
which shows that $F_+\smsh D(F_+)$ is {\it self dual}.
Hence the dualization of the map $d$ above yields
a map
$$
d^*\: S^0 = D(S^0) \to D(F_+ \smsh D(F_+)) \simeq F_+ \smsh D(F_+) \, .
$$
The map $d^*$ is well-defined in the homotopy category of spectra.

Now form the homotopy class
$$
\begin{CD}
t(F)\: S^0 @> d^* >> F_+ \smsh D(F_+) @> \Delta_{F_+} \smsh \text{id} >>
F_+ \smsh F_+\smsh D(F_+) @> \text{id} \smsh d >> 
F_+\smsh S^0\, .
\end{CD}
$$
Then $t(F)$ is identified with a stable homotopy class $S^0 \to Q(F_+)$.
This gives a refined transfer in the case when $B$ is a point. 

Proceeding to the case of a general fibration $E\to B$, 
one appeals to a fiberwise version of the above
to get a refined transfer. As in the introduction,
let $E^+$ denote $E\amalg B$, and define $D_B(E^+)$
to be the fiberwise mapping spectrum of maps $E^+ \to B\times S^0$.
Then analogous to the above, one has a fiberwise stable map
$$
d\: E^+ \smsh_B D_B(E^+) \to B^+
$$
which is adjoint to the identity. 
One then continues in the same
way as above, and the outcome is a 
fiberwise stable homotopy class
$$
B^+ \to E^+ \, .
$$
This gives our rough description of the refined transfer in the general case.

\subsection*{Verification of the axioms} 
The only axiom which is not straightforward 
to verify is the additivity axiom A4. Becker and Schultz
remark that this axiom follows from  formal considerations
involving S-duality. In our context, the crucial points 
are that the map of fibered spectra 
$$
E^+\to D_B D_B(E^+)
$$
is a natural transformation and the
double dual $D_B D_B$ preserves homotopy pushouts.

\section{Characterization when $B$ is a point}
\label{B=pt}

We show how the axioms characterize the refined transfer for
the constant fibration $F \to *$, where $F$ is a homotopy finite cell complex
This case actually follows from the work of Becker and Schultz.
However, it will be useful for what comes later to recast their proof in a 
more coordinate-free language. The case $B = *$ captures
the main features of the proof in the general case.
\medskip

We first digress with an observation
about the axioms in the case of a trivial fibration.
For a trivial fibration 
$$
F \to F \times B \overset{p_B}\to B\, ,
$$
a refined transfer can be regarded as a fiberwise stable 
homotopy class
$$
t(p_B)\: B^+ \to (F_+) \times B \, ,
$$
and the associated 
ordinary transfer can be regarded 
as the associated  stable homotopy class 
$$
t_F(B)\: B_+ \to (F \times  B)_+
$$
which is obtained from $t(p_B)$ by collapsing the preferred
copy of $B$ to a point (note: 
$(F \times  B)_+ = (F_+) \times B)/B$).

More generally, let $(B,A)$ be a cofibration pair.
Choose an actual stable map 
$\hat t(p_B)\: B^+ \to (F_+) \times B$
representing the refined transfer $t(p_B)$.
The naturality axiom implies that the fiberwise homotopy class
of the composite stable map
$$
\begin{CD}
 A^+ @>>> B^+ @> \hat t(p_B) >>  (F_+) \times B 
\end{CD} 
$$
coincides with inclusion  $(F_+) \times A \to (F_+) \times B$
composed with the refined transfer $t(p_A)$ for the trivial fibration
$p_A\: F\times A \to A$. Furthermore, since the diagram
$$
\xymatrix{
F\times A \ar[r] \ar[d] & F\times B\ar[d] \\
A \ar[r] & B
}
$$
is a pullback, the space of choices consisting of 
a choice of representative $\hat t(p_A)$ of $t(p_A)$ together with a 
choice of homotopy
making the diagram of axiom A1  commute is a contractible space.
This shows that once the representative $\hat t(p_B)$ is
chosen, then for any cofibration $A\subset B$,
one obtains a preferred contractible choice of representatives for 
all $t(p_A)$ equipped with a commuting homotopy.

In particular, the representative $\hat t(p_B)$ 
 determines 
a preferred stable homotopy class of pairs
$$
(B_+,A_+) \to ((F \times  B)_+,(F \times  A)_+)\, .
$$
whose components are the transfers $t_F(B)$ and $t_F(A)$.
The latter in turn induces a stable homotopy class on
quotients
$$
t_F(B,A) \: B/A \to (F_+)\smsh (B/A)\, .
$$
Note the special case when $A$ is the empty space gives
$t_F(B,\emptyset) = t_F(B)$.

Axioms A2 and A3 straightforwardly imply
$$
t_F(B,A) \,\, = \,\, t_F\smsh \text{id}_{B/A}\, ,
$$
where $t_F\: S^0 \to F_+$ coincides with 
$t_F(*,\emptyset)$. Also note that $t_F(B,A) = t_F(B/A,*)$.
\medskip

With this observation, we can now  return to the problem
of characterizing the refined transfer when the base space is a point.
In this instance,
a refined transfer is represented by a homotopy class
of stable map $t_F \:S^0 \to F_+$, where $t_F = t_F(*,\emptyset)$.

Because $F$ is a homotopy finite space, there is a 
codimension zero compact {\it smooth} manifold 
$$
M \subset \Bbb R^d
$$
and a homotopy equivalence $F \simeq M$.
By homotopy invariance (i.e., axiom A1 when $f$ is the identity map
of a point), 
it will suffice to characterize
the homotopy class
$$
t_M := t_M(*,\emptyset)\: S^0 \to M_+\, .
$$

Consider
the commutative pullback diagram of pairs
\begin{equation}\label{BS_diag1}
\xymatrix{
(S^d\times M,M) \ar[r]^{\alpha\times 1\quad\,\,} \ar[d]_{\pi_1} & 
((M/\partial M) \times M, M)
\ar[d]^{\pi_1} \\
(S^d,*) \ar[r]_{\alpha} & (M/\partial M,*)
}
\end{equation}
The vertical maps of these diagrams are fibrations.
Applying the relative transfer construction and using
naturality, we see that
the associated diagram of stable maps
\begin{equation}\label{BSdiag2}
\xymatrix{
S^d\smsh M_+ \ar[r]^{\alpha \smsh 1\quad\,\, } & M/\partial M \smsh M_+  \\
S^d \ar[r]_\alpha \ar[u]^{t_M(S^d,*)} & M/\partial M 
\ar[u]_{t_M(M/\partial M,*)}
}
\end{equation}
homotopy commutes. The product axiom for $S^d$ and $F$ implies
that 
$$
t_{*}(S^d,*) \smsh t_M(*,\emptyset) = t_M(S^d,*)\, .
$$
The normalization axiom implies $t_{*}(S^d,*)\: S^d \to S^d$
is the identity, and  $t_M = t_M(*,\emptyset)$ by definition.
Consequently, we see that the diagram of stable maps
\begin{equation} \label{alpha-diag}
\xymatrix{
S^d\smsh M_+ \ar[r]^{\alpha \smsh 1\quad} & M/\partial M \smsh M_+  \\
S^d \ar[r]_\alpha \ar[u]^{1_{S^d}\smsh t_M} & M/\partial M 
\ar[u]_{t_M (M/\partial M,*) \,\, \simeq \,\, 1_{M/\partial M} \smsh t_{M}}
}
\end{equation}
is homotopy commutative.

Consider the diagonal embedding
$$
\Delta\: (M,\partial M) \to (M \times M, (\partial M) \times M).
$$
The associated map of quotients
$M/\partial M \to  (M/\partial M)\smsh (M_+)$ will also be
denoted $\Delta$. The diagonal embedding
has a compact tubular neighborhood isomorphic to the total space of
the unit tangent disk bundle of $M$, which is a trivial
bundle since $M$ is a codimension zero submanifold of euclidean space.
Let $D$ denote the unit disk bundle and let $C$ denote the complement
of the interior of the tubular neighborhood. Then we have a pushout
square
$$
\xymatrix{
S \ar[r]\ar[d]& C\ar[d] \\
D \ar[r] & M\times M\, .
}
$$
The inclusion $(\partial M) \times M \to M \times M$ admits a
factorization up to homotopy through $C$. The factorization
is given by choosing an internal collar of $\partial M$ and
letting $M_0$ denote the result of removing the open collar. 
Then $M_0 \subset M$ is a homotopy equivalence and
the inclusion $\partial M \times M_0 \to M\times M$ has image in $C$
provided that the tubular neighborhood has been chosen sufficiently small.

The inclusion
$$
(M \times M_0,\partial M \times M_0) \to (M \times M,C) 
$$
then gives rise to a map of quotients
$$
M/\partial M \smsh M_+  \cong  M/\partial M_0 \smsh (M_0)_+
\to (M \times M)/C \cong D/S \cong S^d\smsh M_+ \, .
$$
Denote it by  $c\: M/\partial M \smsh M_+\to S^d\smsh M_+$.

\begin{lem}[Compare {\cite[lem.\ 2.5]{Becker-Schultz}}] \label{fact1} The composite
$$
\begin{CD}
S^d \smsh M_+ @> \alpha \smsh 1 >> (M/\partial M) \smsh M_+ 
@> c >> S^d \smsh M_+
\end{CD}
$$
is homotopic to the identity.
\end{lem}

\begin{proof} We can assume that $M$ is embedded in the unit disk $D^d$
in such a way that $M$ meets $\partial D^d = S^{d-1}$ transversely in 
$\partial M$. Then we have an embedding
$(M,\partial M) \subset
(D^d,S^{d-1})$. Let $(W,\partial_0 W)$ be the closure of the complement
of $(M,\partial M)$ in $(D^d,S^{d-1})$. 

Consider the associated embedding
$$
(M \times M,\partial M \times M) \subset (D^d\times M,S^{d-1}\times M) 
$$
given by taking the cartesian product with $M$.
Then the effect of collapsing $W \times M$ in $D^d \times M$ gives
rise to the map $\alpha \smsh 1$.

Consider the  composite
embedding
$$
\begin{CD}
(M,\partial M)@>\Delta >> 
(M \times M,\partial M \times M) \subset (D^d\times M,S^{d-1}\times M)\, .
\end{CD}
$$
The composite $c\circ (\alpha \smsh 1)$ is the effect of collapsing
the complement of a tubular neighborhood $M$ in 
$D^d \times M$ to a point. But the complement is contractible,
so this collapse map is homotopic to the identity.
\end{proof}

\begin{prop}[Compare {\cite[2.9]{Becker-Schultz}}]
\label{fact2} The map $c$ homotopically coequalizes 
$t_M(M/\partial M,*)$ and $\Delta$, i.e.,
$$
c \circ t_M(M/\partial M,*) \,\, \simeq \,\, c\circ \Delta\, .
$$
\end{prop}

\begin{proof}
The argument will use the commutative pushout diagram
of pairs
$$
\xymatrix{
(S,S_0) \ar[r]\ar[d] & (C,C_0)\ar[d]\\
(D,D_0) \ar[r]  &(M\times M,(\partial M)\times M)
}
$$
where $(D,D_0)$ denotes unit tangent disk bundle of $(M,\partial M)$,
$(S,S_0)$ is the unit sphere bundle and $(C,C_0)$ is the complement.
Let $$\pi_1\: (M\times M,(\partial M)\times M) \to (M,\partial M)$$ be the
first factor projection. Let $p_D \: (D,D_0) \to (M,\partial M)$
be its restriction to $(D,D_0)$; this is fibration pair
with fiber $D^d$.  Similarly, let 
$p_C\: (C,C_0) \to (M,\partial M)$ and $p_S\: (S,S_0) \to (M,\partial M)$
be the restrictions to $(C,C_0)$ and $(S,S_0)$ respectively. 
Each of these is also a fibration pair. 
The fiber of $p_S$ is $S^{d-1}$ and
the fiber of $p_C$ is $M_0$, where
$M_0$ is the effect of removing an open ball from the interior of $M$. We
therefore have a pushout square of fibers
$$
\xymatrix{
S^{d-1} \ar[r] \ar[d] & M_0\ar[d] \\
D^d \ar[r] & M \, .
}
$$
Then the additivity axiom implies
$$
t_M(M,\partial M) = j_1 t_{D^d}(M,\partial M) + j_2 t_{M_0}(M,\partial M)
-  j_{12} t_{S^{d-1}}(M,\partial M) \, ,
$$
where $j_S$ for $S\subset \{1,2\}$ is induced by the evident
inclusion map into $M$. 

Then by the homotopy invariance 
and normalization axioms  
$$
t_{D^d}(M,\partial M) = i\circ t_*(M,\partial M) = i \, ,
$$
where $i\:M/\partial M\to D/D_0$ arises from the zero section.
By definition,  $j_1i$ is the reduced diagonal map
$\Delta\: M/\partial M\to M/\partial M \smsh M_+$. Consequently, 
$$
j_1t_{D^d}(M,\partial M) = \Delta\, .
$$
In particular, 
$$c\circ j_1t_{D^d}(M,\partial M) = c\circ \Delta\, .$$

To complete the proof of the proposition, it will suffice 
to show that  $c$ applied to each of the
terms $j_2 t_{M_0}(M,\partial M)$ and
$j_{12} t_{S^{d-1}}(M,\partial M)$ is trivial, for this will 
yield  $c \circ t_M(M,\partial M) = c\Delta$.

To see why $c\circ j_2 t_{M_0}(M,\partial M)$ is trivial, recall
that $c$ is defined by collapsing $C \subset M\times M$ to a point, 
whereas  $j_2t_{M_0}(M,\partial M)$ is given by a composite of the form
$$
\begin{CD}
M/\partial M @>t_{M_0}(M,\partial M) >> 
C/C_0 @>>> (M\times M)/(\partial M \times M)\, .
\end{CD}
$$
The triviality of $cj_2 t_{M_0}(M,\partial M)$ therefore follows
from the fact that it factors through $C/C_0$. A similar argument shows that
$cj_{12} t_{S^{d-1}}(M,\partial M)$ is trivial.
\end{proof}

\begin{proof}[Proof of Theorem \ref{first} when $B$ is a point]
By \ref{fact1} and \ref{fact2} and diagram \eqref{alpha-diag} we have
$$
t_M \,\, = \,\,  c \circ (\alpha\smsh 1) \circ t_M \,\, = \,\, 
c\circ t_M(M/\partial M,*) \circ \alpha \,\, = \,\, 
c\circ \Delta \circ \alpha \, .
$$
This shows that $t_M$ is determined by the maps $c,\Delta$ and $\alpha$
whose definition is independent of $t_M$.
\end{proof}

\section{Interpretation}\label{interp}
Motivated by Peter May's paper on the Euler characteristic
in the setting of  derived categories \cite{May_additivity},
we give an alternative interpretation of what we have just shown
in terms of the algebra of the stable homotopy category.
This section is independent of the rest of the paper.

Given $F$ as above, recall that
$D(F_+)$ denotes the $S$-dual of
of $F_+$. Then $D(F_+)$ is a ring spectrum with unit
$u\: S^0 \to D(F_+)$ which represents 
a desuspension of the map $\alpha$ appearing above. 

In what follows we  consider  $F_+$ 
as an object of the stable homotopy category. Then we have an action
$$
{\mu}\: D(F_+) \smsh F_+ \to F_+
$$
which is defined as the 
formal adjoint to the map $D(F_+) \to \hom(F_+,F_+)$
given by mapping a stable map $f\:F_+ \to S^0$ to $f\smsh 1_{F_+}$
composed with the diagonal of $F_+$. The map ${\cal \mu}$ is a
homotopy theoretic version of the collapse map $c$ 
described above. Lemma \ref{fact1} in this language
asserts that $F_+$ is a $D(F_+)$-module.

Furthermore,  $F_+$ is a coalgebra in the stable category,
and one has a co-action map
$$
\kappa \: D(F_+)\to D(F_+) \smsh F_+
$$
which expresses $D(F_+)$ as an $F_+$-comodule: it
can be defined as the linear dual of $\mu$
(= maps into $S^0$). Then
$\kappa$ is a homotopy theoretic version of the diagonal
map $\Delta$.

\begin{cor}\label{stable-cor}
With respect to above, we have:
\begin{itemize}
\item the diagram
$$
\xymatrix{
S^0\smsh F_+ \ar[r]^{u\smsh 1\quad } & D(F_+) \smsh F_+ \\
S^0 \ar[r]_{u\qquad} \ar[u]^{t_F} & D(F_+)\smsh S^0\ar[u]_{1_{D(F_+)}\smsh t_F}
}
$$
is homotopy commutative (cf.\
diagram \eqref{alpha-diag});
\item The composite
$$
\begin{CD}
S^0\smsh F_+ @> u\smsh 1>>  D(F_+) \smsh F_+ @>\mu >> F_+ =  S^0 \smsh F_+
\end{CD}
$$
is homotopic to the identity;
\item The map $\mu$ homotopically coequalizes $1_{D(F_+)}\smsh t_F$ and
$\kappa$ (cf.\ \ref{fact2}).
\end{itemize}
\end{cor}

From \ref{stable-cor} we immediately infer
$$
t_F \,\, = \,\, \mu\circ \kappa \circ u \, .
$$

\section{Proof of Theorem \ref{first}}\label{smoothings}

Let $p\: E\to B$ be a fibration with homotopy finite
fibers. Assume
$p$ admits a compact TOP reduction
$q\: W \to B$.
When $B$ is homotopy finite, one can replace $q$ with its
topological stable normal bundle along the fibers to obtain a new 
compact TOP reduction which is a
codimension zero subbundle
of the trivial bundle $B\times \Bbb R^j$ for $j$ sufficiently large
(cf.\ \cite[p.\ 599]{Becker-Schultz}, \cite{Rourke-Sanderson}).

Consequently, we can assume without loss in generality
that $q$ comes equipped with
a fiberwise codimension zero topological embedding
 $W\subset B \times \Bbb R^d$. We let
$$
\partial^v W \to B
$$
be the {\it fiberwise boundary} of $q$. This is the bundle
whose fiber at $b \in B$ is given by $\partial W_b$.

The idea of the proof of Theorem \ref{first}
will be to adapt the method of  \S \ref{B=pt} 
to the fiberwise topological setting.

The  proof will hinge upon the following
structure, which is assumed to vary continuously in $b\in B$:
\begin{itemize}
\item The fibers $W_b$ come equipped with a degree one
collapse map $S^d \to W_b/\partial W_b$ and
\item The diagonal map $(W_b,\partial W_b) \to
(W_b\times W_b ,(\partial W_b)\times W_b)$ has
a compact tubular neighborhood.
\end{itemize}
The first of these properties is given by taking the
Thom-Pontryagin collapse of the embedding $W_b \subset \{b\}\times \Bbb R^d$.
The second property is discussed in \cite[p.\ 599]{Becker-Schultz}.

\begin{proof}[Proof of Theorem \ref{first}]

As in \S \ref{B=pt}, the proof begins by considering a
commutative diagram (the fibered analogue of diagram \eqref{BS_diag1}):
$$
\xymatrix{
(S^d\times W,W) \ar[rr]^{(\alpha \times_B 1,1)\qquad} 
\ar[d]_{(1\times q,q)} && 
((W/\!\!/\partial^v W)\times_B W,W)
\ar[d]^{(q^*,q)} \\
(S^d\times B,B) \ar[rr]_{(\alpha,1)}&&(W/\!\!/\partial^v W,B) \, .
}
$$
Here $W/\!\!/\partial^v W$ denotes the 
pushout of the diagram $B \leftarrow \partial^v W \subset W$,
and 
$q^*$ denotes the pullback of $q\: W\to B$ along 
the map $W/\!\!/\partial^v W\to B$. Note that the fibers of
 $W/\!\!/\partial^v W\to B$ are given by $W_b/\partial W_b$.
The map  $\alpha$ is given by the fiberwise Thom-Pontryagin collapse
map of the codimension zero embedding $W\subset \Bbb R^d$.

Given a refined transfer $t$, we apply it to the above
 and appeal to the naturality
and product axioms to obtain a homotopy commutative diagram
$$
\xymatrix{
\Sigma^d_B W^+ \ar[r]^{\alpha \smsh_B 1\qquad} & 
(W/\!\!/\partial^v W) \smsh_B W^+  \\
\Sigma^d_B B^+\ar[r]_\alpha \ar[u]^{\Sigma^d_B t(q)} & W/\!\!/\partial^v W 
\ar[u]_{t(q^*,q)}\, ,
}
$$
where $\Sigma^d_B W^+$ denotes the $d$-fold fiberwise suspension
of $W^+ \to B$ (note that $\Sigma^d_B B^+ = B \times S^d$). 
The vertical maps are refined transfer maps associated with
fibration pairs. This is the fibered analog of diagram \eqref{alpha-diag}.

Let 
$$
c\: (W/\!\!/\partial^v W) \smsh_B W^+ \to \Sigma^d_B W^+
$$
be the fiberwise collapse map
which on each fiber, maps the complement data of the fiberwise
embedding of the diagonal to a point.

To complete the proof, we appeal to the following two assertions,
which are fiberwise versions of \ref{fact1} and \ref{fact2}, and are proved
similarly (alternatively, these are proved in
the smooth case in \cite[lem.\ 2.5, eq.\ 2.9]{Becker-Schultz}
and their proof adapts in our case).
\medskip

{\flushleft \it Assertion 1.}  The composition 
$
c\circ (\alpha \smsh_B 1)\: \Sigma^d_B W^+ \to \Sigma^d_B W^+
$
is fiberwise homotopic to the identity.
\medskip 

{\flushleft \it Assertion 2.} 
The map $c$ homotopically coequalizes
the $t(q^*,q)$ and the fiberwise reduced diagonal map
$\Delta\: W/\!\!/\partial^v W  \to  (W/\!\!/\partial^v W) \smsh_B W^+$.
That is, $c\circ t(q^*,q) \simeq c\circ \Delta$.
\medskip

Given these assertions, we obtain the equation
$$
\Sigma^d_B t(q)\,\, \simeq \,\, c\circ \Delta\circ \alpha\, ,
$$
which uniquely determines $t(q)$.
\end{proof}

\section{Proof of Theorems \ref{second} and \ref{split}}

We first give the definition of the assembly map.
Given a homotopy functor 
$$
f\:\text{Top} \to \text{Spectra}
$$
from spaces to spectra, the assembly map is
an associated natural transformation of homotopy functors
$$
f^\%(X) \to f(X)\, 
$$
giving the best approximation to
$f$ on the left by a homology theory. The simplest
construction of  $f^\%(X)$ is to take
the homotopy colimit
$$
\underset{\Delta^k \to X}{\text{ hocolim }} f(\Delta^k)
$$
where the indexing is given by the category
whose objects are the singular simplices $\Delta^k\to X$ and
whose morphisms are given by restriction to faces.

The natural transformation $f^\%(X) \to f(X)$ is induced from the evident
maps $f(\Delta^k) \to f(X)$ associated with each
singular simplex $\Delta^k \to X$. 
Using the weak equivalence $f(\Delta^k) \to f(*)$, we obtain
a weak equivalence of functors
$$
f^\%(X) \overset\sim\to X_+\smsh f(*) \, .
$$
Consequently, we
may think of the assembly map as a natural transformation
$$
X_+ \smsh f(*) \to f(X) \, .
$$
For more details, see \cite{Weiss-Williams_assembly}.

\begin{proof}[Proof of  Theorem \ref{split}] We will give two proofs.
The first, suggested by a referee, shows that the assembly map
is $(2r-c)$-split for some constant $c\ge 0$.

Recalling the decomposition
$$
A(X) \,\, \simeq \,\, \Sigma^\infty (X_+) \times 
\text{\it Wh}^{\text{diff}}(X)
$$
(\cite{Wald_concord}),
it will suffice to split the assembly map of each factor.
The assembly map for $\Sigma^\infty (X_+)$ is the identity map, so
it clearly splits. The assembly map for $\text{\it Wh}^{\text{diff}}(X)$
has the form
$$
X_+ \smsh \text{\it Wh}^{\text{diff}}(*) \to \text{\it Wh}^{\text{diff}}(X)\, .
$$
There is a natural map
$$
\text{\it Wh}^{\text{diff}}(*) 
\to 
X_+ \smsh \text{\it Wh}^{\text{diff}}(*)
$$
which arises from the basepoint of $X$.
The composite with the assembly map yields
the map
$$
\text{\it Wh}^{\text{diff}}(*) \to \text{\it Wh}^{\text{diff}}(X)
$$
that is induced by the  inclusion of the basepoint.
As noted by Waldhausen \cite[p.\ 153]{Wald_manifold}, for
$r$-connected $X$, this map is approximately $2r$-connected.
This completes the first proof.

Our second proof shows that the
assembly map for $A(X)$ is $2r$-split.
It uses the commutative diagram
of spectra
\begin{equation} \label{Good_diagram}
\xymatrix{
A(X) \ar[r] \ar[d] & X_+\smsh S^0\ar[d] \\
A(*) \ar[r] & S^0
}
\end{equation}
where the vertical maps are induced by the map $X\to *$,
and the horizontal ones are given using Waldhausen's
splitting of $A(X)$ mentioned above.

Suppose that $X$ is $r$-connected. Then 
Goodwillie has shown that the square \eqref{Good_diagram} 
is $(2r+1)$-cartesian (\cite[cor.\ 3.3]{Good_calc1}). 
It follows that the square
$$
\xymatrix{
X_+ \smsh S^0 \ar[r] \ar[d] & A(X)\ar[d] \\
S^0 \ar[r] & A(*)
}
$$
is $2r$-cartesian, where the horizontal maps
are given by the natural transformation
from stable homotopy to the algebraic $K$-theory of spaces
induced from the inclusion functor from 
finite sets to finite spaces.

Using the basepoint for $X$,
it follows that
$$
X\smsh S^0 \to A(X) \to A(*)
$$
is a homotopy fiber sequence up through dimension $2r$
in the sense that the map from $X\smsh S^0$ to the homotopy fiber
of the map $A(X) \to A(*)$ is  $2r$-connected.
Furthermore, we have a homotopy fiber sequence
$$
A(*) \smsh X \to A^\%(X) \to A(*)\, .
$$

Consider the diagram
$$
\xymatrix{
X\smsh A(*) \ar[r]\ar[d] &  A^\%(X) \ar[r]\ar[d]  & A(*)
\ar@{=}[d] \\
 X\smsh S^0\ar[r]&  A(X)  \ar[r] & A(*) 
}
$$
where the middle vertical map is the assembly map, and
the left vertical map is given by smashing the 
map $A(*) \to S^0$ with the identity map of $X$.
Since the lower row is a homotopy fiber
sequence up through dimension $2r$, and
the map $A(*) \smsh X \to S^0 \smsh X$ is 
homotopically split, it follows that the assembly map
$$
A^\%(X)\to A(X)
$$
is $2r$-split. The functor $B(X)$ in this case is given by the wedge
$$
(X\smsh S^0) \vee A(*)\, ,
$$
and the map $B(X) \to A^\%(X)$ is given using
the evident map 
$X\smsh S^0 \to X\smsh A(*) \simeq  A^\%(X)$ together
with the map $A(*) \to A^\%(X)$ coming from the basepoint
of $X$.
\end{proof}

\begin{rem} There is no such
 stable splitting of the assembly map for $A(X)$
on the category of {\it un}based spaces. For suppose there were a homotopy functor $B(X)$ defined on unbased spaces equipped with a natural transformation 
$B(X) \to A^\%(X)$ such that the composite $B(X) \to A^\%(X) \to A(X)$
is $(2r-c)$-connected for $r$-connected $X$. Taking the 
first stage of the Goodwillie tower of these functors yields
maps
$$
B(X) \to P_1B(X) \to P_1A^\%(X) \to P_1 A(X) 
$$
such that the composite $B(X) \to P_1A(X)$ is a weak equivalence.
Since $A^\%(X) = P_1A^\%(X)$, we infer that 
$A^\%(X) \to P_1A(X)$ has a section up to homotopy. 
But this is impossible when $X = \emptyset$ is the empty space, since
$A^\%(\emptyset) = *$ whereas $P_1 A(\emptyset)$ is not contractible.
We are indebted to a referee for explaining this argument to us.
\end{rem}

\begin{proof}[Proof of Theorem \ref{second}]
Consider the fibration 
$p\: E \to B$ with section whose underlying map is $(r+1)$-connected. Then
the fiberwise version of the assembly map
$$
A^\%_B(E) \to A_B(E)
$$
is defined and is
a map of fibered spectra over $B$ 
(cf.\  \cite[p.\ 51]{Dwyer-Weiss-Williams}). Applying
the method of our second proof of Theorem \ref{split}  in a fiberwise manner,
we obtain a fibered spectrum ${\cal B}_B(E)$ and a map
${\cal B}_B(E)\to A^\%_B(E)$ such that the composite
$$
{\cal B}_B(E)\to A^\%_B(E) \to A_B(E)
$$
is $2r$-connected (note: the fiber of ${\cal B}_B(E)$ at a point
$b\in B$ is identified with $B(E_b)$).

By slight abuse of notation, consider the fiberwise
assembly map as a map of (associated) fiberwise infinite
loop spaces. Then assuming that $B$ has the homotopy
type of a cell complex of dimension $\le 2r$, 
the previous paragraph shows that induced map of section spaces
$$
\text{sec}(A^\%_B(E) \to B) \to 
 \text{sec}(A_B(E)\to B)
$$
is surjective on path components. 

By the ``Converse Riemann-Roch Theorem'' of 
Dwyer, Weiss and Williams (\cite[cor.\ 10.18]{Dwyer-Weiss-Williams}),
it follows that $p\: E\to B$ admits a compact TOP reduction.
\end{proof}

\section{Proof of Theorems \ref{third} and \ref{third_add}} \label{D}

\begin{proof}[Proof of Theorem \ref{third}] 
Recall
that Waldhausen's space $A(*)$ has the same 
rational homotopy type as $K(\Bbb Z)$,
the algebraic $K$-theory space of the integers. 
Borel showed that the rational cohomology
of $K(\Bbb Z)$ is an exterior algebra on classes
$b_{4k+1}$ in degree $4k+1$ for $k > 0$ \cite{Borel}.
Therefore, $H^{4k+1}(A(*);{\Bbb Z})$ has 
a nontrivial torsion free summand for $k>0$.

Let $\vee_k S^n$ be a $k$-fold wedge of $n$-spheres and let 
${\cal H}_k^n$ denote
the topological monoid of based homotopy equivalences of $\vee_k S^n$ (cf. \cite[p.\ 385]{Wald_LNM}. 
Then we have a homomorphism ${\cal H}_k^n\to {\cal H}_k^{n+1}$ given by 
suspension and a homomorphism ${\cal H}_k^n \to {\cal H}_{k+1}^n$ given
by wedging on a single copy of the identity map. One of the standard
definitions of $A(*)$ is 
$$
{\Bbb Z} \times \lim_{k,n} B {\cal H}_k^n \, {}^+ \, ,
$$
where  ``$B$'' denotes the classifying space functor, and ``${}^+$'' denotes
Quillen's plus construction. In particular, the natural  map 
$$
\iota\: {\Bbb Z} \times \lim_{k,n} B {\cal H}_k^n \to 
A(*) ,
$$
(from the space to its plus construction) is a rational 
cohomology isomorphism. 
Let $x\in H^{4i+1}(A(*);{\Bbb Z})$ be any non-torsion element. Then the restriction
$\iota^*x \in H^{4i+1}(B {\cal H}_k^n;{\Bbb Z})$ is also non-torsion when $k$ and
$n$ are chosen sufficiently large. Finally, let 
$$
B \subset B {\cal H}_k^n
$$
be a connected, finite subcomplex which supports the class 
$\iota^* x$. This inclusion can be thought
of as a classifying map for a fibration 
$$
p\:E \to B
$$
whose fiber at the basepoint is $\vee_k S^n$. 

Theorem \ref{third} will now follow from:

\begin{cl} The fibration $p\: E\to B$ does not admit a compact fiber smoothing. However,
$p$ admits a compact {\rm TOP} reduction 
provided $n$ is sufficiently large.
\end{cl}

The second part of the claim follows directly from Theorem \ref{second}. 
To prove the first part, 
it will be sufficient by the work of Dwyer, Weiss and Williams  to 
prove that the {\it $A$-valued trace map}
$$
\chi_A(p)\:B \to A(*)
$$
(cf.\ below) does {\it not} admit a factorization up to homotopy as 
$$
B\to Q(S^0)\to A(*)\, .
$$
This is sufficient because the theory (cf.\ \cite[\S12]{Dwyer-Weiss-Williams}) 
shows that the existence
of a  compact fiber smoothing would imply the existence of
 such a factorization.

Since $Q(S^0)$ has trivial rational cohomology in positive degrees, it
suffices to show that the $A$-valued trace is rationally non-trivial
in cohomology in degree $4i+1$.

We first give a  quick sketch of the construction of the 
$A$-valued trace map using 
the alternative definition of $A(*)$ as the algebraic $K$-theory of the category
of homotopy finite based spaces (with cofibrations and weak equivalences). Deferring
to Waldhausen's notation, let $wR^{\text{hf}}(*)$ be the category
 whose objects  are based, homotopy finite 
cofibrant topological spaces, and whose morphisms
are weak homotopy equivalences. In particular,
a homotopy finite space $F$ determines an object of
$wR^{\text{hf}}(*)$, namely $F_+$.

Waldhausen also gives a `1-skeleton' inclusion map 
$$
j\: |wR^{\text{hf}}(*)| \to A(*)
$$
(\cite[p.\ 329]{Wald_LNM}),
so we can regard $F_+$ as point of either the realization $|wR^{\text{hf}}(*)|$ or of $A(*)$. 
Apply this construction to each fiber of the fibration $p$. This gives
for each $x\in B$, a point $(F_x)_+ \in A(*)$ which can be arranged 
so as to vary continuously in $x$ (the reader is referred to 
\cite[\S1.6]{Dwyer-Weiss-Williams}
for the details).  This yields the desired trace map 
$\chi_A(p)\:B \to A(*)$.

From the construction we have given, it is immediate
that have a factorization of $\chi_A(p)$ as
$$
\begin{CD}
B @> v >> |wR^{\text{hf}}(*)| @>j >> A(*)\, ,
\end{CD}
$$
where $v$ is a continuous rectification of the map $x\mapsto (F_x)_+$.

On the other hand, since the fibration $p\: E \to B$ comes equipped
with a preferred section (arising from the wedge point), each fiber
$F_x$ is automatically a based space. We therefore have another
object $F_x \in wR^{\text{hf}}(*)$, so we have another map
$$
u\:B \to |wR^{\text{hf}}(*)|
$$
given by $x \mapsto F_x$. 
Waldhausen has also shown (\cite[prop.\ 2.2.5]{Wald_LNM}) 
that 
the component of $|wR^{\text{hf}}(*)|$ which contains 
the fibers $F_x$ is homotopy
equivalent to the classifying space $B{\cal H}_k^n$, and with respect
to this identification,
$u$ can be regarded as the classifying map of the fibration $p$.
In particular, the composition
$$
\begin{CD}
B @>\subset >> B{\cal H}_k^n  \to |wR^{\text{hf}}(*)| @>j >> A(*)
\end{CD}
$$
coincides with $j\circ u$ up to homotopy and is 
therefore rationally non-trivial in degree $4i+1$.

Now, using Waldhausen's additivity theorem (\cite[prop.\ 1.3.2]{Wald_LNM}),
$j\circ v$
coincides with the map $x \mapsto (F_x) \vee S^0$, which is wedge sum
of the map $j\circ u$ and the constant map with value $S^0 \in A(*)$,
and recall that wedge sum gives the $H$-space structure on $A(*)$
 (\cite[p.\ 330]{Wald_LNM}). 
In particular, the maps $j\circ v$ and $j\circ u$ coincide on
rational cohomology in positive degrees, so $j\circ v$ is rationally
non-trivial in degree $4i+1$.

This completes both the proof of the claim and also the proof
of Theorem \ref{third}.
\end{proof} 

\begin{proof}[Proof of Theorem \ref{third_add}]
Waldhausen \cite[\S3]{Wald_manifold} constructs a map
$$
BF \to A(*)
$$
where $F$ is the topological monoid of based 
stable self homotopy equivalences of the sphere. 

Let $F_k$ denote the topological monoid of
based (unstable) self homotopy equivalences of $S^k$.
Then, up to homotopy, the composite
$$
BF_k \to BF  \to A(*)
$$
can be conveniently described as follows. Recall that
$$
wR^{\text{hf}}(*)
$$
is the category of weak equivalences of cofibrant based spaces.
Let $wR^{\text{hf}}(*)_{(S^k)}$ denote the component
of $wR^{\text{hf}}(*)$ that contains the sphere $S^k$.
Then the realization
$$
|wR^{\text{hf}}(*)_{(S^k)}|
$$ 
has the homotopy type of $BF_k$ (\cite[prop.\ 2.2.5]{Wald_LNM}), 
and with respect to this identification, the map
$BF_k \to A(*)$ is  the `1-skeleton'
inclusion 
$$
j\: |wR^{\text{hf}}(*)_{(S^k)}| \subset A(*)
$$
(cf.\  \cite[p.\ 329]{Wald_LNM}).

B\"okstedt and Waldhausen \cite[p.\ 419]{Bok_Wald} 
have shown that the composite
map
$$
BF \to A(*) \to \text{\it Wh}^{\text{diff}}(*)
$$
is non-trivial on homotopy groups in degree three. The
second map in the composite is the splitting map for
$A(*) \simeq Q(S^0) \times \text{\it Wh}^{\text{diff}}(*)$.

By the Freudenthal suspension theorem,
the  homomorphism $$\pi_3(BF_3) \to \pi_3(BF) = \Bbb Z_2$$
is an isomorphism. On the level of spherical
fibrations, this group is generated by
the clutching construction of a map $S^2 \times S^3 \to S^3$
whose associated Hopf construction $S^5 \to S^3$ represents $\eta^2$,
where $\eta \in \pi_1^{\text{st}}(S^0)$ is represented
by the Hopf map $S^3 \to S^2$. The clutching construction
produces the spherical
fibration $S^3 \to E \to S^3$ stated in Theorem \ref{third_add}.

It follows from the computation of B\"okstedt and 
Waldhausen  that the image of
this generator under the homomorphism
$$
\pi_3(BF_3) \to \pi_3(A(*))
$$
is not an element of the subgroup 
$$
{\Bbb Z}_{24} = \pi_3(QS^0) \subset \pi_3(A(*))\, .
$$
From the theory of Dwyer, Weiss and Williams 
\cite[\S12]{Dwyer-Weiss-Williams}, we infer
the fibration fails to have a compact fiber smoothing. 
On the other hand, the fibration
is admits a compact TOP reduction by Theorem \ref{second}. Since this fibration
represents a torsion element, it is not detected rationally.
\end{proof}

\section{Proof of Theorem \ref{trace}} \label{E}

Let $t(p),t'(p)\: B^+\to E^+$ be refined transfers associated
to a fibration $p\: E\to B$. We will show that the traces 
$$
\text{tr}_t(p),\text{tr}_{t'}(p)\: B_+\to S^0
$$ 
coincide. Let $S_Bp\:S_B E\to B$ be the fiberwise suspension of $p$.
Let $C_B p\: C_B E\to B$ be the mapping cone of $p$. Then the inclusion
$B \to C_B E$ is a fiber homotopy equivalence.

Apply the additivity and normalization axioms to the pushout
$$
\xymatrix{
E \ar[r]\ar[d]& C_B E \ar[d]\\
C_B E \ar[r] &  S_B E
}
$$
and then take the associated traces to get
$$
\text{tr}_t (S_B p) = 1 + 1 - \text{tr}_t (p) \, ,
$$
where $1\: B_+ \to S^0$ is the unit map.

Applying fiberwise suspension again, we obtain
$$
\text{tr}_t (S^2_B p) = \text{tr}_t (p) \, .
$$
Iterating this last equation $j$-times, we get
$$
\text{tr}_t (S^{2j}_B p) = \text{tr}_t (p)\, .
$$
If $j$ is sufficiently large, the fibration $S^{2j}_B p$
admits a compact TOP reduction by 
Theorem \ref{second}. By Theorem \ref{first},  $t$ and $t'$ agree on $S^{2j}_B p$. 
Taking traces we conclude $\text{tr}_t (p)= \text{tr}_{t'} (p)$.

\section{Proof of Theorem \ref{kill}} 

Let $m$ be the dimension of the finite complex $X$.
Let $X^k$ denote the $k$-skeleton, and let 
$X^{(k)}$ denote the quotient $X^k/X^{k-1}$.

Consider the fibration $p\: E\to B$.
At each fiber $E_x$, there is a cofibration sequence
of retractive spaces over $E_x$ of the form
$$
(E_x \times X^k) \amalg E_x  \to (E_x \times X^{k+1}) \amalg E_x 
\to E_x \times X^{(k+1)}\, ,
$$
for $k \ge 0$. Denote the sum operation in 
the category of retractive spaces by $+$; this is given by
fiberwise wedge.
By the additivity theorem \cite[prop.\  1.3.2]{Wald_LNM},
we obtain a preferred homotopy class of path in $A(E_x)$
from the sum
$$
(E_x \times X^{k})\amalg E_x \,\, +  \,\, (E_x \times X^{(k+1)})
$$
to $(E_x \times X^{k+1}) \amalg E_x $. 

Summing over $k$, we get a preferred homotopy
class of path in $A(E_x)$ connecting the points 
\begin{equation}\label{alternating}
(E_x \times X) \amalg E_x \qquad \text{ and } \qquad
\sum_{k=0}^m  E_x \times  X^{(k)}\, .
\end{equation}
Let $T_k$ denote a based set having cardinality one more
than the number of $k$-spheres in $X^{(k)}$. Then we get
an identification $T_k \smsh S^k \cong  X^{(k)}$.

As fiberwise suspension
induces the homotopy inverse to the $H$-multiplication
defined by the sum (see  \cite[prop.\ 1.6.2]{Wald_LNM}),
and the sum operation is homotopy commutative, there
is a preferred path between the above and the sum
$$
\sum_{k}  E_x \times  T_{2k} \,  + \, 
\sum_{k}  E_x \times  \Sigma T_{2k+1} \, .
$$
The latter can be rewritten as
$$
E_x \times (T_0 \vee T_2 \vee \cdots) \,\, + \,\, 
E_x \times \Sigma (T_1 \vee T_3 \vee \cdots)
$$
where the based sets $(T_0 \vee T_2 \vee \cdots)$ and
$(T_1 \vee T_3 \vee \cdots)$ have the same cardinality 
under the assumption that the Euler characteristic of $X$ is
zero.
Consequently, if we let $T$ denote $(T_0 \vee T_2 \vee \cdots)$,
the above is identified with 
$$
(E_x\times T) \,\, \vee \,\, \Sigma_{E_x} (E_x \times T)
$$
which, by the additivity theorem has a preferred homotopy
class of path to the zero object.

Now the assignment $x\mapsto (E_x \times X) \amalg E_x$ gives
rise to the generalized Euler characteristic
of the fibration $q\: E\times X \to B$,
which is a section of the fibration $A_B(E\times X) \to B$
(\cite[I.1]{Dwyer-Weiss-Williams}). The above argument
shows that this section is vertically homotopic to
the constant section given by the basepoint of each fiber 
$A(E_x)$. But the basepoint section clearly 
factors through the map $Q_BE\to A_B E$ via the
basepoint section of the fibration $Q_BE\to B$. 
We now apply the converse Riemann-Roch theorem in the smooth
case (\cite[\S12]{Dwyer-Weiss-Williams})  to conclude that 
$q$ admits a compact fiber smoothing.

This completes the proof of Theorem \ref{kill}.

\begin{rem} When $X = (S^1)^{\times k}$ is a torus of sufficiently
large dimension, Theorem \ref{kill} becomes the `closed fiber
smoothing theorem' of Casson and Gottlieb \cite[p.\ 160]{Casson-Gottlieb}.
\end{rem}

\section{Proof of Theorem \ref{decomp}}
\label{TtoE}

By replacing $M$ by $M \times D^2$ if necessary, we can assume that
 $M \subset \Bbb R^m$ is a codimension zero compact connected smooth submanifold
such that $\partial M \subset M$ is $2$-connected.

Let ${\cal E}(M,*)$ denote the geometric realization of
the simplicial monoid whose $k$-simplices are  families of topological embeddings
$$
e\:\Delta^k \times M \to \Delta^k \times M
$$ 
such that $e$ commutes with projection to $\Delta^k$, 
$e$ is a homotopy equivalence and is the identity
when restricted to $\Delta^k \times *$.

Similarlty, let $\text{TOP}(M,*)$ be the geometric realization of the 
simplicial group whose $k$-simplices are self homeomorphism of $\Delta^k \times M$ that
preserve $\Delta^k \times \ast$.
Then one has a forgetful homomorphism
$$
\text{TOP}(M,*) \to  {\cal E}(M,*)
$$
of topological monoids
which induces a map of classifying spaces
$$
B\text{TOP}(M,*) \to B{\cal E}(M,*)\, ,
$$
whose homotopy fiber  is identified with
the Borel construction
$$
E\text{TOP}(M,*) \times_{\text{TOP}(M,*)} {\cal E}(M,*) \, .
$$
The latter may also be identified
the orbit space ${\cal E}(M,*)/\text{TOP}(M,*)$, because the action of
$\text{TOP}(M,*)$ on ${\cal E}(M,*)$ is free.

The orbit space may also be  identified with
${\cal H}(\partial M)$, the space of topological $h$-cobordisms of 
$\partial M$.  This can be seen as follows: let ${\cal E}'(M,*)$
be defined just as ${\cal E}(M,*)$ but where we now require the embedding
$e$ to have image in $\Delta^k \times \text{int}(M)$, where $\text{int}(M)$ is the interior of
$M$. Using a choice of collar neighborhood of $\partial M$, one sees that the inclusion
${\cal E}'(M,*) \subset {\cal E}(M,*)$ is a 
deformation retract. Therefore, the orbit space is also identified with 
the Borel construction of $\text{TOP}(M,*)$ acting on ${\cal E}'(M,*)$.
Define a map  ${\cal E}'(M,*) \to {\cal H}(\partial M)$ by sending
an embedding $e\:\Delta^k\times  M \to \Delta^k\times  \text{int}(M)$
to the $k$-parameter family of $h$-cobordisms 
$$
(\Delta^k \times M) \,\, - \,\, e(\Delta^k \times \text{int}(M)) \, .
$$
Then
$$
\text{TOP}(M,*) \to {\cal E}'(M,*) \to  {\cal H}(\partial M)
$$
is a homotopy fiber sequence (compare \cite[p.\ 170]{WW-Auto}), so the
assertion follows.

Taking classifying spaces, we extend to the right
to obtain a homotopy fiber sequence
$$
{\cal H}(\partial M) \to B\text{\rm TOP}(M,*) \to B{\cal E}(M,*)\, .
$$
Let $B =  B{\cal E}(M,*)$, and let $E\to B$ be the associated universal
fibration with fiber $M$, obtained as follows: 
the tautological action of ${\cal E}(M,*)$ on $M$ gives a Borel construction 
$E{\cal E}(M,*) \times_{{\cal E}(M,*)} M \to B{\cal E}(M,*)$ which is a
quasifibration. Then $E \to B$ is the effect of converting the quasifibration
 into a fibration.

Using this fibration, 
obtains a fiberwise
generalized Euler characteristic
$$
B \overset \chi\to  A_{B}(E)\, .
$$
The restriction
of $\chi$ to $B\text{TOP}(M,*)$ factors through the fiberwise
assembly map
$$
 A^\%_{B}(E) \to A_B(E)
$$
via an ``excisive characteristic'' $\chi^\%$ 
(\cite[7.11]{Dwyer-Weiss-Williams}; here
we are considering the fiberwise assembly map as a map
of fiberwise infinite loop spaces).
The resulting diagram
\begin{equation} \label{dummies}
\xymatrix{
B\text{TOP}(M,*) \ar[rr]^{\chi^\%} \ar[d] && A^\%_B(E) \ar[d] \\
B \ar[rr]_\chi && A_B E \, .
}
\end{equation}
is preferred homotopy commutative.

Taking vertical homotopy fibers, we obtain a map of spaces
$$
{\cal H}(\partial M) \to \Omega \text{\it Wh}^{\text{top}}(M)\, .
$$
This map is a composite of the form
$$
{\cal H}(\partial M) \overset{(a)}\to 
\Omega \text{\it Wh}^{\text{top}}(\partial M) \overset{(b)}\to
\Omega \text{\it Wh}^{\text{top}}(M)\, ,
$$
where the map $(a)$ is an equivalence in
the topological concordance stable range (which is
approximately $m/3$ by \cite{Igusa_stability} and
\cite{Burghelea-Lashof_stability}). 
In particular, by taking
the cartesian product of $M$ with a disk of sufficiently large dimension, 
we can assume
that the map $(a)$ is highly connected. The map $(b)$ is
is induced by applying the functor $\Omega \text{\it Wh}^{\text{\rm top}}$
to the inclusion $\partial M \to M$. If we replace $M$ by
$M \times D^k$, it too becomes highly connected when $k$-grows
because $\partial (M \times D^k) \to M \times D^k$
is at least $(k-1)$-connected, and after applying the functor 
the result is also approximately $k$-connected (since the
same is true for the functor $A^\%$ and also the functor $A$ using, say, 
\cite[cor.\ 3.3]{Good_calc1}).
The upshot of this is that we
can, by replacing $M$ by $M \times D^k$, assume the composite
map $(b)\circ (a)$ a weak equivalence up through any given 
dimension. We will assume this to be the case.

Now let $M$ be $r$-connected.
It follows with respect to our assumptions
that the diagram \eqref{dummies} 
is $2r$-cartesian. By the method of proof of Theorem \ref{second},
we know that the  fiberwise assembly map 
$A^\%_B(E) \to A_B(E)$ is $2r$-split in a preferred way. 
This shows that the map
$$
B\text{TOP}(M,*) \to B{\cal E}(M,*)
$$
is also $2r$-split.
It follows that we have a preferred decomposition of
homotopy groups
$$
\pi_*(\text{TOP}(M,*)) \,\, \cong \,\, \pi_*({\cal E}(M,*)) \oplus
\pi_{*+1}(\text{\it Wh}^{\text{\rm top}}(M))
$$
for $* < 2r-1$. 

To complete the proof we need to identify ${\cal E}(M,*)$.
Let ${\cal I}(M,*)$ be the realization of the simplicial
monoid defined just as  ${\cal E}(M,*)$ but now with immersions
in place of embeddings.
By topological transversality \cite{Kirby-Sieben}, the inclusion map 
$$
{\cal E}(M,*) \to  {\cal I}(M,*)
$$
is a weak equivalence in our range after replacing $M$ with $M \times D^k$ for
$k$ sufficiently large. Finally, let $\tau_M$ be the 
topological tangent microbundle of $M$, which is
a trivial fiber bundle $M \times {\Bbb R^m} \to M$
 since $M$ is a codimension zero submanifold of euclidean space. 
Let $G(\tau_M,*)$ be the (simplicial) monoid whose zero simplices
are pairs $(f,\phi)$ such that $f\: M \to M$ is a 
based self homotopy equivalence
and $\phi\: \tau_M \to \tau_M$ is a fiber bundle isomorphism covering $f$.
The $k$-simplices of $G(\tau_M)$ are families of such pairs parametrized by
the standard $k$-simplex. Then we have an identification
$$
G(\tau_M,*) \,\, =  \,\, G(M,*) \times \text{maps}(M,\text{TOP}_m) \, ,
$$
and the evident map
$$
{\cal I}(M,*) \to G(\tau_M,*) \, ,
$$ is known to be a weak equivalence by immersion theory
\cite[p.\ 137]{Lashof}. Assembling this information 
completes the proof of Theorem \ref{decomp}.

\begin{rem} A more careful statement of
Theorem \ref{decomp} is  as follows. Let $M$ had dimension $m$ and spine dimension $d$.
Let $c$ be the concordance stable range of $M$ (this is the connectivity of
the stabilization map $C(M) \to C(M\times I)$, where $C(M)$ is the smooth concordance 
space of $M$; by \cite{Igusa_stability} one has  $c \ge \max(2m+7,3m+4)$).
Then the map
$$
B\text{\rm TOP}(M,\ast) \to BG(M,\ast)
$$
has a section up to homotopy
 on the $(2r)$-skeleton provided that both $m-d$ and $c$ are greater than
$2r$. Consequently, if the homotopy type of $M$ is held fixed, one needs the dimension of $M$
to approximately exceed both $6r$ and $d+2r$ for there to be a section.
\end{rem}

\section{Appendix: characteristic classes for fibrations}
\label{appendix}

This section, which might be of independent interest, sketches a theory of 
characteristic classes for fibrations with homotopy finite base and fibers.
These classes were implicitly used in section \ref{D}.

Let $B$ be a connected finite complex. Then
as in section \ref{D}, a fibration $p\: E\to B$
with homotopy finite fibers gives an
$A$-valued trace map
$$
\chi_A(p) \:B\to A(*) \, .
$$
Pulling back the Borel classes $y_{4k+1} \in H^{4k+1}(A(*); \Bbb Q)$,
we obtain rational cohomology classes
$$
y_{4k+1}(p) \in H^{4k+1}(B;\Bbb Q), \qquad k > 0\, .
$$
These classes vanish whenever $p$ admits a compact fiber smoothing.
Furthermore, they satisfy the following axioms:
\begin{itemize}
\item (Naturality). The classes $y_{4k+1}(p)$ are natural with respect to 
base change.
\item (Products). For a product fibration 
$p\times p'\: E\times E' \to B\times B'$  with fiber $F\times F'$,
we have
$$
y_{4k+1}(p\times p') \,\, = \,\, 
y_{4k+1}(p)\otimes \chi(F') \,\, + \,\,  \chi(F)\otimes y_{4k+1}(p')\, ,
$$
where $\chi(F) \in H^0(B) \cong {\Bbb Z}$ is the Euler characteristic. 
\item (Additivity). If 
$$
\xymatrix{
E_{\emptyset} \ar[r]\ar[d] & E_2\ar[d]\\
E_1 \ar[r] & E
}
$$
is a homotopy pushout of fibrations over $B$ having homotopy
finite fibers, then 
$$
y_{4k+1}(p) \,\, =\,\,  y_{4k+1}(p_1) + y_{4k+1}(p_2) - y_{4k+1}(p_{12})\, ,
$$
where $p_S \: E_S \to B$ for $S\subsetneq \{1,2\}$. 
\end{itemize}

\begin{rems} (1). The classes $y_{4k+1}(p) \in H^{4k+1}(B;\Bbb Q)$ are primary obstructions
to finding a compact fiber smoothing. When there is a
compact fiber smoothing, one has the higher Reidemeister torsion classes 
$$
\tau_{4k}(p) \in H^{4k}(B;\Bbb Q)
$$
defined by Igusa \cite{Igusa_book}. One can view the latter
as a corresponding theory of secondary characteristic classes
of the fibration $p$ that depend on the specific choice
of compact fiber smoothing.
\smallskip

{\flushleft (2).}  Given $p\: E \to B$, let $q\: E\times X\to B$
be the effect of taking cartesian product a map $X\to *$ where $X$ is 
a finite complex  having zero Euler characteristic. 
Then $y_{4k+1}(q)$ vanishes by the product axiom (compare 
Theorem \ref{kill}).
\end{rems}


\begin{thebibliography}{DWW}
\bibliographystyle{invent}
\bibliography{john}


\bibitem[BD]{Badzioch-Dora}%
Badzioch, B., Dorabia\l a, W.:
Additivity for parametrized 
topological Euler characteristic and Reidemeister torsion.
\newblock arXiv preprint math.AT/0506258




\bibitem[BG1]{Becker-Gottlieb}%
Becker, J.C., Gottlieb, D.H.: 
Transfer maps for fibrations and duality.
\newblock {\it Comp.\ Math.\ \rm}
{\bf 33}, 107-133 (1976)

\bibitem[BG2]{Becker-Gottlieb_adams}%
Becker, J.C., Gottlieb, D.H.: The transfer map and fiber bundles.
\newblock {\it Topology\ \rm}
{\bf 14}, 1--12 (1975)

\bibitem[BS]{Becker-Schultz}%
Becker, J.C., Schultz, R.E.: 
Axioms for bundle transfers and traces.
\newblock {\it Math.\ Zeit.\ \rm}
{\bf 227}, 583--605 (1998)

\bibitem[BW]{Bok_Wald}%
B\"okstedt, M., Waldhausen F.\
\newblock The map $BSG\to A(*)\to QS\sp 0$.  
Algebraic topology and algebraic $K$-theory (Princeton, N.J., 1983),  
418--431, 
Ann. of Math. Stud., 113, Princeton Univ. Press, Princeton, NJ, 1987.


\bibitem[B]{Borel}%
Borel, A.: Cohomologie r\'eelle stable 
de groupes $S$-arithmétiques classiques.
\newblock {\it C.R.\ Acad.\ Sci.\ Paris S\'er. A-B} {\bf 274},
A1700--A1702 (1972)

\bibitem[BL]{Burghelea-Lashof_stability}%
Burghelea, D., Lashof, R.:
Stability of concordances and the suspension homomorphism.
\newblock {\it Ann.\ of Math.} {\bf 105}, 449--472 (1977)


\bibitem[BM]{Brumfiel-Madsen}%
Brumfiel, G., Madsen, I.: 
Evaluation of the transfer and the universal surgery classes.  
\newblock{\it Invent.\ Math.} {\bf 32}, 133--169 (1976)

\bibitem[Ch]{Chapman_invariance}%
Chapman, T.A.: Topological invariance of Whitehead torsion.  
\newblock{\it Amer.\ J.\ Math.} {\bf 96},  488--497 (1974)


\bibitem[Co]{Cohen_simple}%
Cohen, M.M.: A course in simple-homotopy theory. 
\newblock Graduate Texts in Mathematics, Vol.\ 10. Springer-Verlag, 1973.


\bibitem[CG]{Casson-Gottlieb}%
Casson, A., Gottlieb, D.H.: Fibrations with compact fibres.  
\newblock{Amer.\ J.\ Math.} {\bf 99} 159--189 (1977)

\bibitem[CJ]{Crabb-James}%
Crabb, M., James, I.: Fibrewise homotopy theory. 
\newblock Springer Monographs in Mathematics, 1998

\bibitem[D]{Dora_coset}%
Dorabia\l a, W.: The double coset theorem formula 
for algebraic $K$-theory of spaces. 
\newblock {\it $K$-Theory} {\bf 25}, 251--276 (2002) 


\bibitem[DJ]{Dora-Johnson}%
Dorabia\l a, W., Johnson, M.:
Factoring the Becker-Gottlieb Transfer Through the Trace Map.
\newblock arXiv preprint math.KT/0601620

\bibitem[D]{Douglas}%
Douglas, C.L.:
Trace and transfer maps in the algebraic $K$-theory of spaces
\newblock{\it $K$-Theory} {\bf 36},  59--82 (2005/2006). 
 

\bibitem[DWW]{Dwyer-Weiss-Williams}%
Dwyer, W., Weiss, M., Williams, B.:
A parametrized index theorem for the algebraic $K$-theory Euler class. 
\newblock {\it Acta Math.\ \rm}  {\bf 190}, 1--104 (2003) 


\bibitem[G]{Good_calc1}%
Goodwillie, T.G.; Calculus I, 
The first derivative of pseudoisotopy theory.
\newblock{\it $K$-theory} {\bf 4},
1--27 (1990)

\bibitem[H]{Hatcher}%
Hatcher, A.E.:
Concordance spaces, higher simple-homotopy theory, and applications. 
\newblock Algebraic and geometric topology 
(Proc. Sympos. Pure Math., Stanford Univ., Stanford, Calif., 1976), 
Part 1, pp. 3--21,
\newblock Proc. Sympos. Pure Math., XXXII, Amer. Math. Soc., 1978.

\bibitem[I1]{Igusa_book}%
Igusa, K.: Higher Franz-Reidemeister torsion.
\newblock AMS/IP Studies in Advanced Mathematics, 31. 
Amer.\ Math.\ Soc., 2002

\bibitem[I2]{Igusa_axioms}%
Igusa, K.: Axioms for higher torsion invariants of smooth bundles.
\newblock: arXiv preprint math.KT/0503250



\bibitem[I3]{Igusa_stability}%
Igusa, K.: The stability theorem for smooth pseudoisotopies.  
\newblock {\it $K$-Theory} {\bf 2}, no. 1-2 (1988) 




\bibitem[KS]{Kirby-Sieben}%
Kirby, R.C., Siebenmann, L.C.: 
Foundational essays on topological manifolds, smoothings, and triangulations. 
\newblock Annals of Mathematics Studies, No. 88. 
Princeton University Press, 1977

\bibitem[L]{Lashof}%
Lashof, R.: The immersion approach to triangulation and smoothing.
\newblock Algebraic Topology (Proc. Sympos. Pure Math., Vol. XXII, 
Univ. Wisconsin, Madison, Wis., 1970),  
131--164.\  Amer. Math. Soc., 1971 

\bibitem[M]{May_additivity}%
May, J. P.: The additivity of traces in triangulated categories.  
\newblock {\it Adv.\ Math.}  {\bf 163},  34--73 (2001) 

\bibitem[MS]{May-Sigurdsson}%
May, J.P., Sigurdsson J.: {Parametrized homotopy theory}.
\newblock Mathematical Surveys and Monographs, vol. 132,
\newblock Amer. Math. Soc., 2006


\bibitem[Qu]{Quillen}%
Quillen, D.: {Homotopical Algebra}.
\newblock (LNM, Vol.~43). Springer 1967

\bibitem[RS]{Rourke-Sanderson}%
Rourke, C.P., Sanderson, B.J.: On topological neighbourhoods.  
\newblock {\rm Compositio Math.}  {\bf 22},  387--424 (1970)


\bibitem[S]{Schwede}%
Schwede, S.: {Spectra in model categories and applications to the algebraic
  cotangent complex}.
\newblock {\it J. Pure Appl. Algebra\rm } {\bf 120}, 77--104 (1997)

\bibitem[St]{Strom}%
Str{\o}m, A.: {The homotopy category is a homotopy category}.
\newblock {\it Arch. Math. (Basel)\rm} {\bf 23}  435--441 (1972) 


\bibitem[W1]{Wald_LNM}%
Waldhausen, F.: Algebraic $K$-theory of spaces.
\newblock {\it Algebraic and Geometric Topology, Proceedings Rutgers 1983,
LNM 1126}, 1985, pp.\ 318--419.

\bibitem[W2]{Wald_manifold}%
Waldhausen, F.: Algebraic $K$-theory of spaces, a manifold approach.  
\newblock Current trends in algebraic topology, Part 1 (London, Ont., 1981),  
\newblock pp. 141--184, CMS Conf.\ Proc., 2, Amer.\ Math.\ 
Soc., Providence, R.I., 1982.

\bibitem[W3]{Wald_concord}%
Waldhausen, F.: Algebraic $K$-theory of spaces, concordance, 
and stable homotopy theory.  
\newblock Algebraic topology and algebraic $K$-theory 
(Princeton, N.J., 1983),  392--417, Ann. of Math. Stud., 113, 1987 



\bibitem[WW1]{Weiss-Williams_assembly}%
Weiss, M; Williams, B.:  Assembly.  
Novikov conjectures, index theorems and rigidity, Vol. 2 (Oberwolfach, 1993),  
\newblock 332--352, LMS Lecture Notes 227, 
Cambridge Univ. Press, 1995


\bibitem[WW2]{WW-Auto}%
Weiss, M; Williams, B.: Automorphisms of Manifolds.
\newblock Surveys in Surgery Theory, Vol. 2, Ann. of Math. Studies 149
\newblock Princeton University Press, 2001


\end{thebibliography}
\end{document}